\documentclass[letterpaper, 10 pt, conference]{ieeeconf}  

\IEEEoverridecommandlockouts                              
\usepackage{graphicx} 
\usepackage{amsmath} 
\usepackage{amssymb}  
\usepackage{array}
\usepackage{multirow}

\newtheorem{thm}{Theorem}
\newtheorem{proposition}[thm]{Proposition}
\newtheorem{lemma}[thm]{Lemma}

\newtheorem{assumption}{Assumption}

\title{A Sufficient Condition for Partial Ensemble Controllability of Bilinear 
Schr\"{o}dinger Equations with Bounded Coupling Terms}

\author{
\authorblockN{Thomas Chambrion}
\authorblockA{Universit\'e de Lorraine, Institut \'Elie Cartan de Lorraine, UMR 7502,  
Vand{\oe}uvre-­l\`es-­Nancy,F-­54506, France  \\
CNRS, Institut \'Elie Cartan de Lorraine, UMR 7502, Vand{\oe}uvre-­l\`es-­Nancy, F-­54506,
France\\
Inria, CORIDA, Villers-l\`es-Nancy, F-54600, France\\
{\tt\small Thomas.Chambrion@univ-lorraine.fr}}
}

\begin{document}

\maketitle
\thispagestyle{empty}
\pagestyle{empty}

\begin{abstract}
This note presents a sufficient condition for partial approximate ensemble controllability  
of a set of bilinear conservative systems in an infinite dimensional Hilbert space. The 
proof relies on classical geometric and averaging control techniques applied on finite 
dimensional approximation of the infinite dimensional system. The results are illustrated 
with the planar rotation  of a linear molecule. 
\end{abstract}

\section{INTRODUCTION}

\subsection{Control of quantum systems}
The state of a quantum system evolving in a Riemannian manifold $\Omega$ is
described by 
its \emph{wave function}, a point $\psi$ in $L^2(\Omega, \mathbf{C})$. When the
system is 
submitted to an electric field (e.g., a laser), the time evolution of the wave
function is 
given, under the dipolar approximation and neglecting decoherence,  by the
Schr\"{o}dinger 
bilinear 
equation:
\begin{equation}\label{EQ_bilinear}
\mathrm{i} \frac{\partial \psi}{\partial t}=(-\Delta  +V(x)) \psi(x,t) +u(t) 
W(x) 
\psi(x,t)
\end{equation}
where $\Delta$ is the Laplace-Beltrami operator on $\Omega$,  $V$ and $W$ are
real 
potential accounting for the properties of the free system and the control
field 
respectively, while 
the real function of the time $u$ accounts for the intensity of the laser. 

In view of applications (for instance in NMR), it is important to know whether and how
it is possible to chose a suitable control $u:[0,T]\to \mathbf{R}$ in order to steer 
(\ref{EQ_bilinear}) from a given initial state to a given target. This question has raised 
considerable interest in the community in the last decade. After the negative results of 
\cite{bms} and \cite{Turinici} excluding 
exact controllability on the natural domain of the operator $-\Delta +V$ when
$W$ is bounded, the first, and 
at this day the only one, description of the attainable set for an example of
bilinear quantum system 
was obtained by Beauchard (\cite{beauchard,beauchard-coron}).  Further investigations of
the approximate controllability of (\ref{EQ_bilinear}) were conducted using
Lyapunov techniques  
(\cite{nersesyan, Nersy, beauchard-nersesyan,Mirrahimi, MR2168664, mirra-solo}) 
and geometric techniques (\cite{Schrod,Schrod2}).
\subsection{Ensemble controllability}

In many applications, a macroscopic device acts on a large number of identical microscopic 
quantum systems (for instance, a single laser acts on a small quantity of liquid 
containing many molecules). Usually, the external field acts differently on each of the 
small systems (depending for instance on the orientation of the molecule with respect to 
the exterior electric field). For the sake of simplicity, we will assume in this work that 
the action of the external field on the system $a$ is proportional to $a$.   Instead of 
one system of type (\ref{EQ_bilinear}), one has to control a continuum:
\begin{equation}\label{EQ_bilinear_continuum}
\mathrm{i} \frac{\partial \psi_\alpha}{\partial t}=(-\Delta  +V(x) +u(t) 
\alpha W(x)) \psi_\alpha(x,t) 
\end{equation}
where the system labeled with $\alpha$, $\alpha\in [0,1]$, has wave function $\psi_\alpha$. Notice that, 
since the systems are physically identical, the free dynamics (when $u=0$) is the same for 
every $a$.

The  simultaneous (or \emph{ensemble}) control problem turns into the following question: 
let a continuum of 
initial conditions $(\psi^0_a)_{a \in [0,1]}$ and of targets $(\psi^t_a)_{a \in [0,1]}$ be 
given. Does it exist a control $u$ that steers the systems (\ref{EQ_bilinear_continuum}) 
from $\psi^0_a$ to $\psi^t_a$ for every $a$ in $[0,1]$? 

Because of its crucial importance for applications (dispersion of parameters is 
\emph{always} present in real world systems), the problem of \emph{ensemble controllability} 
of quantum systems has been tackled by many authors, see for instance   
\cite{PhysRevA.73.030302, MR2879412,MR2608124,4177466,MR2191545,SarletteEnsemble} for 
theoretical results and \cite{6425988} for numerical aspects.

\subsection{Framework and notations}

To take advantage of the powerful tools of linear operators, we will reformulate the 
problem (\ref{EQ_bilinear_continuum}) in the following abstract setting. Let $H$ be a 
separable Hilbert space, endowed with the $\langle \cdot,\cdot 
\rangle$ Hilbert product. We consider the continuum of control systems
\begin{equation}\label{EQ_main}
\frac{\mathrm{d}\psi}{\mathrm{d}t}=A\psi +u(t) \alpha B\psi, \quad \alpha \in [0,1],
\end{equation}
where the linear 
operators $A,B$ can be completed in a 5- or 6-uple that satisfies Assumption \ref{ASS_AB} 
or Assumption \ref{ASS_AB_spectre_discret}.
\begin{assumption}\label{ASS_AB}
The 6-uple $(A,B,\lambda_1,\phi_1,\lambda_2,\phi_2)$  satisfies
\begin{enumerate}
\item $A$ is skew-adjoint with domain $D(A)$;\label{ASS_A_skew_adjoint}
\item $B$ is bounded and skew-symmetric; \label{B_bounded}
\item $\phi_1$ and $\phi_2$ are two eigenvectors of $A$ of norm 1, associated with  eigenvalues $-\mathrm{i}\lambda_1$ and $-\mathrm{i}\lambda_2$;
\label{ASS_A_eigenvectors}
\item $\lambda_1 <\lambda_2$ and $\langle \phi_1,B\phi_2\rangle \neq 0$;
\item for every eigenvalues $-\mathrm{i}\mu$ and $-\mathrm{i}\mu'$ of $A$, associated 
with eigenvectors $v$ and $v'$, $|\lambda_1-\lambda_2|=|\mu-\mu'|$ implies 
$\{\lambda_1,\lambda_2\}=\{\mu,\mu'\}$ or $\{\lambda_1,\lambda_2\}\cap \{\mu,\mu'\}=\emptyset$ 
or $\langle v,B v'\rangle=0$;\label{ASS_non_degenerate}
\item  the essential spectrum of $\mathrm{i}A$ (if any) does not accumulate in any of 
these four points: $ 2\lambda_1-\lambda_2, \lambda_1, \lambda_2$ and 
$2\lambda_2-\lambda_1$.\label{ASS_sspectre_essentiel}
\end{enumerate}
\end{assumption}

\begin{assumption}\label{ASS_AB_spectre_discret}
The 5-uple $(A,B,U,\Lambda,\Phi)$  satisfies
\begin{enumerate}
\item $A$ is skew-adjoint with domain $D(A)$;\label{ASS_A_skew_adjoint_discret}
\item $B$ is skew-symmetric; \label{B_skew_symmetric}
\item $U$ is a subset of $\mathbf{R}$ containing at least $0$ and the points 
$\{1/n, n\in \mathbf{N}\}$;\label{ASS_U_accumul}
\item for every $u$ in $U$, $A+uB$ is skew-adjoint (with domain not necessarily equal 
to $D(A)$);\label{ASS_AuB_skew_adjoint}
\item $\Phi=(\phi_j)_{j\in \mathbf{N}}$ is a Hilbert basis of $H$ made of  eigenvectors 
of $A$, all of which in the domain of $B$;
\label{ASS_Phi_eigenvectors}
\item $\Lambda=(\lambda_j)_{j\in \mathbf{N}}$ is a sequence of real numbers such that, 
for every $j$ in $\mathbf{N}$, $A\phi_j=-\mathrm{i}\lambda_j$; \label{ASS_Lambda_eigenvalues}
\item $\lambda_1 <\lambda_2$ and $\langle \phi_1,B\phi_2\rangle \neq 0$;\label{ASS_couplage_spectre_discret}
\item for all eigenvalues $-\mathrm{i}\mu$ and $-\mathrm{i}\mu'$ of $A$, associated 
with eigenvectors $v$ and $v'$, $|\lambda_1-\lambda_2|=|\mu-\mu'|$ implies 
$\{\lambda_1,\lambda_2\}=\{\mu,\mu'\}$ or $\{\lambda_1,\lambda_2\}\cap \{\mu,\mu'\}=\emptyset$ 
or $\langle v,B v'\rangle=0$.\label{ASS_non_degenerate_discret}
\end{enumerate}
\end{assumption}
If $(A,B)$ satisfies Assumptions \ref{ASS_AB}.\ref{ASS_A_skew_adjoint} and 
\ref{ASS_AB}.\ref{B_bounded} (resp. $(A,B,U)$ satisfies Assumptions 
\ref{ASS_AB_spectre_discret}.\ref{ASS_A_skew_adjoint_discret}, 
\ref{ASS_AB_spectre_discret}.\ref{ASS_U_accumul} and 
\ref{ASS_AB_spectre_discret}.\ref{ASS_AuB_skew_adjoint}), then for every $t_0,t$ 
in $\mathbf{R}$,  for 
every $u$ in $L^1(\mathbf{R},\mathbf{R})$ (resp. $u:\mathbf{R}\to U$ piecewise constant),  
there exists a unique family of unitary operators 
$(\Upsilon^{u,\alpha}_{t,t_0})_{\alpha \in [0,1]}$ such that, for every family 
$(\psi^0_\alpha)_{\alpha \in [0,1]}$ in $H$, for every $\alpha$ in $[0,1]$, $t\mapsto 
\Upsilon^{u,\alpha}_{t,t_0}\psi^0_\alpha$ is the unique solution of (\ref{EQ_main}) 
in the weak sense that 
satisfies $\Upsilon^{u,\alpha}_{t_0,t_0}\psi^0_\alpha=\psi^0_\alpha$.

Let $(A,B,\lambda_1,\phi_1,\lambda_2,\phi_2)$   satisfy Assumption \ref{ASS_AB} or 
$(A,B,U,\Lambda,\phi)$ satisfy Assumption \ref{ASS_AB_spectre_discret}. We 
define the 2-dimensional Hilbert space ${\mathcal L}_2=\mathrm{span}(\phi_1,\phi_2)$, 
$\pi_2:H\to H$ the orthogonal 
projection on ${\mathcal L}_2$, $A^{(2)}=\pi_2A\pi_2$ and $B^{(2)}=\pi_2A\pi_2$ the 
compressions of $A$ and $B$ on ${\mathcal L}_2$, and $X_{(2)}^{u,\alpha}$, the propagator 
associated with the (infinite dimensional) system $x'=A^{(2)}x+u\alpha B^{(2)}x$. By 
abuse, we will still denote with $X_{(2)}^{u,\alpha}$ the restriction of 
$X_{(2)}^{u,\alpha}$ to ${\mathcal L}_2$.

\subsection{Main result}

\begin{proposition} \label{PRO_main_result}
Let $(A,B,\lambda_1,\phi_1,\lambda_2,\phi_2)$  satisfy Assumption \ref{ASS_AB}  (resp.
$(A,B,U,\Lambda,\Phi)$ satisfy Assumption \ref{ASS_AB_spectre_discret}) and let
$\hat{\Upsilon}:\alpha \in [0,1]\mapsto  \hat{\Upsilon}^\alpha \in U({\mathcal L}_2)$ be 
a continuous curve of unitary operators on ${\mathcal L}_2$ that satisfies 
$\hat{\Upsilon}^0=\mathrm{Id}_{{\mathcal L}_2}$. Then, for every $\varepsilon >0$, for 
every $\delta>0$, there exists $u:[0,T]\to [-\delta,\delta]$ (resp. piecewise constant 
with value in $U \cap [0,\delta]$) such that 
for every $\alpha$ in $[0,1]$ and $(j,k)$ in $\{1,2\}\times \{1,2\}$,
$\left ||\langle \phi_j, \Upsilon^{u,\alpha}_{T,0}\phi_k \rangle|-|\langle \phi_j, 
\hat{\Upsilon}^\alpha \phi_k \rangle| \right |<\varepsilon$.
\end{proposition}

In other words, up to an arbitrary small error $\varepsilon$, it is possible to steer the 
eigenvectors $\phi_1$ and $\phi_2$, simultaneously for every $\alpha$, to a target with 
prescribed modulus of coordinates on $\phi_1$ and $\phi_2$.

The contribution of this note relies on the very same idea as \cite{PhysRevA.73.030302}, 
namely the computation of finite dimensional Lie brackets  and a polynomial 
interpolation. The only novelty is that all the steps of the proof come along with 
explicit estimates, which allow to let the dimension of the finite dimensional systems 
tend to infinity  and eventually prove the infinite dimensional  result. 

The main improvements of this work with respect to the cited references are
\begin{itemize}
\item the possibly infinite dimension of the ambient space $H$;
\item the possibility for the spectrum of $A$ to have a continuous part;
\item the possible (finite or not) degeneracy (or multiplicity) of the eigenvalues 
of $A$;
\item (in the case of Assumption \ref{ASS_AB_spectre_discret}) the possible unboundedness 
of operator $B$ with respect to $A$, that is, in a case where Kato-Rellich theorem does not apply to $A+uB$. 
\end{itemize}
\subsection{Content of the paper}

The core of the proof of Proposition \ref{PRO_main_result} is a constructive approximate 
controllability result about the propagator $X_{(n)}^{u,\alpha}$ in some finite dimensional subspaces 
$\mathcal{L}_n$ of $H$ proved in Section \ref{SEC_Finite_dimension}. The precise estimates 
of Section \ref{SEC_averaging} allows to let the dimension of $\mathcal{L}_n$ tend to 
infinity and eventually to prove, in Section \ref{SEC_spectre_discret}, the infinite 
dimensional result for systems satisfying Assumption \ref{ASS_AB_spectre_discret}. In 
Section \ref{SEC_spectre_continu}, we will see that the convergence process used for the 
proof of section \ref{SEC_spectre_discret} is actually robust enough with respect to 
perturbation of the spectrum of $A$ to ensure convergence also for systems satisfying 
Assumption \ref{ASS_AB}. 
The results are applied to the example of the 3D rotation of a 
collection of linear molecules in Section \ref{SEC_Example}.

\section{FINITE DIMENSIONAL PRELIMINARIES}\label{SEC_Finite_dimension}

\subsection{Notations and result}
Let $N$ in $\mathbf{N}$, $A^{(N)},B^{(N)}$ be two matrices in $\mathfrak{u}(N)$ (that is, 
$\overline{A^{(N)}}^T + A^{(N)}= \overline{B^{(N)}}^T + B^{(N)}=0$). We consider the 
continuum of $N$-dimensional systems 
\begin{equation}\label{EQ_main_dim_fin}
x'=A^{(N)}+u\alpha B^{(N)}x, \quad \alpha \in[0,1]
\end{equation}
where $x$ is a point in $\mathbf{C}^N$ endowed with its canonical Hilbert structure 
$\langle \cdot,\cdot,\rangle$. 
For every locally integrable function $u$, we define $X^{u,\alpha}_{(N)}$ the propagator 
associated with (\ref{EQ_main_dim_fin}). 

We assume that $A^{(N)}$ is diagonal in 
$(\phi_j)_{j\leq N}$, the canonical 
basis of $\mathbf{C}^N$,  we denote with $(-\mathrm{i}\lambda_j)_{j\leq N}$ the diagonal 
of $A^{(N)}$ and with $b_{jk}:=\langle \phi_j,B^{(N)}\phi_k\rangle , 1\leq j,k\leq N$ 
the entries of $B^{(N)}$. For every $j\leq N$, we define $\pi_j$, the orthogonal projection of $\mathbf{C}^N$ to ${\mathcal L}_j=\mathrm{span}(\phi_1,\ldots,\phi_j)$. 

\begin{proposition}\label{PRO_main_result_dim_fin}
Assume that $(A^{(N)},B^{(N)},\lambda_1,\phi_1,\lambda_2,\phi_2)$ satisfies 
Assumption \ref{ASS_AB}.
 Let $\hat{\Upsilon}:\alpha \in [0,1]\mapsto  \hat{\Upsilon}^\alpha \in 
 U({\mathcal L}_2)$ be 
a continuous curve of unitary operators on ${\mathcal L}_2$ that satisfies 
$\hat{\Upsilon}^0=\mathrm{Id}_{{\mathcal L}_2}$. Then, for every $\varepsilon >0$, for 
every $\delta>0$, there exists $u:[0,T]\to [-\delta,\delta]$ such that, for $(j,k)$ in 
$\{1,2\}^2$,
$\left ||\langle \phi_j, \Upsilon^{u,\alpha}_{T,0}\phi_k \rangle|-|\langle \phi_j, 
\hat{\Upsilon}^\alpha \phi_k \rangle| \right |<\varepsilon.$
\end{proposition}

The proof of Proposition \ref{PRO_main_result_dim_fin} is split in two steps. In a first 
time, after a suitable change of variable,  we introduce a continuum of two-dimensional  
auxiliary  systems in  Section \ref{SEC_auxiliary_system}. Classical Lie groups technique, 
and the associated uniform convergence estimates, to prove approximate ensemble controllability of  these systems.
In a second time, in Section \ref{SEC_averaging}, we use classical averaging techniques  
to show that the trajectories of the systems introduced in Section 
\ref{SEC_auxiliary_system} can be tracked, with arbitrary precision, by the system  
(\ref{EQ_main_dim_fin}).

\subsection{An auxiliary system}\label{SEC_auxiliary_system}
We consider the continuum of control systems in $\mathbf{U}(2)$
\begin{equation}\label{EQ_syst_aux_avat_chgt_variable}
x_{\alpha}'=\alpha \left ( \begin{array}{ll} b_{11} & b_{12}e^{\mathrm{i}\theta}\\
{b_{21}}e^{-\mathrm{i}\theta} & b_{22} \end{array} \right ) x_{\alpha}, \quad \alpha \in [0,1]
\end{equation}
with initial condition $x_{\alpha}(0)=I_2$ and control function 
$\theta:\mathbf{R}\to \mathbf{R}$. For every piecewise constant function 
$\theta:\mathbf{R}\to \mathbf{R}$, for every $\alpha$, we denote with 
$Y^{\theta,\alpha}$ the propagator of (\ref{EQ_syst_aux_avat_chgt_variable}).

\begin{proposition}\label{PRO_contr_syst_aux}
Let $\hat{\Upsilon}:\alpha \in [0,1]\mapsto  
\hat{\Upsilon}^\alpha \in SU({\mathcal L}_2)$ be 
a continuous curve of unitary operators on ${\mathcal L}_2$ that satisfies 
$\hat{\Upsilon}^0=\mathrm{Id}_{{\mathcal L}_2}$. Then, for every $\varepsilon >0$, there exists $u:[0,T]\to [-\pi,\pi]$ piecewise constant such that, for $(j,k)$ in 
$\{1,2\}^2$,
$$\left ||\langle \phi_j, Y^{\theta,\alpha}_T\phi_k \rangle|-|\langle \phi_j, 
\hat{\Upsilon}^\alpha \phi_k \rangle| \right |<\varepsilon.$$
\end{proposition}
\begin{proof}
Let $\varepsilon>0$. There exists a continuous function $\tilde{\upsilon}:\alpha\mapsto 
\tilde{\upsilon}^\alpha \in \mathfrak{su}(\mathcal{L}_2)$ such that, for every $\alpha$ in 
$[0,1]$, $\|\exp(\tilde{\upsilon}^\alpha)-\hat{\Upsilon}^\alpha\|<\varepsilon$. By density 
of odd polynomials mapping, for the norm of uniform convergence, in the set of odd 
continuous functions, there exists a polynomial mapping $P:\alpha\mapsto 
P_\alpha=\sum_{l=0}^N \alpha^{2l+1} Z_{2l+1}$, with $Z_1,Z_3,\ldots, Z_{2N+1}$ in 
$\mathfrak{su}(2)$ such that $\|P_\alpha-\tilde{\upsilon}^\alpha\|<\varepsilon$ for every 
$\alpha$ in $[0,1]$.

\begin{lemma}\label{LEM_suivi_polynomes_crochets}
Let $X$ and $Y$ two matrices in $\mathfrak{su}(2)$, and $(C_j(X,Y))_{1\leq j\leq p}$ a 
sequence of iterated brackets of $X$ and $Y$. We denote with $l_j$ the length of the 
bracket $C_j(X,Y)$ (the length of $[X,Y]$ is 1). Then, for every real sequence 
$(\beta_j)_{1\leq j\leq p}$, for every $T$ in $\mathbf{R}$, for every $\varepsilon>0$, 
there exists a finite sequence $(t_k)_{1\leq k \leq m}$ in $\mathbf{R}$ such that, for 
every $\alpha$ in $[0,1]$, 
$\|P_{\alpha}-e^{T \sum_{j=1}^p \beta_j \alpha^{l_j}C_j(X,Y)}\|<\varepsilon$, 
where $P_{\alpha}$ is the
 product of matrices $e^{t_1 \alpha X}e^{t_2 \alpha Y}\cdots e^{t_{m-1}\alpha X}e^{t_m \alpha Y}$.
\end{lemma}
\begin{proof}
This result is very classical when $\alpha=1$ (i.e., one considers one system only). The   
uniform version presented here (with $\alpha$ in $[0,1]$)  is basically contained in 
\cite{MR2191545}. Because of its importance for our purpose, we give below a sketch of the 
proof of the result.

We first assume that $p=1$ and we proceed by induction on the length of $C_1$. 
From the Baker-Campbell-Hausdorff formula, we deduce that, for every $2\times 2$ matrices $X,Y$, there exists a function $g_{X,Y}:\mathbf{R}\to \mathfrak{gl}_2$ tending to $0_{2\times 2}$ at 0 such that, for every $t$ in $\mathbf{R}$, for every $\alpha$ in $[0,1]$,
$$ e^{t\alpha X}e^{t\alpha Y}e^{-t \alpha X}e^{-t \alpha Y}=e^{\alpha ^2 t^2[X,Y]+g_{X,Y}(\alpha t)\alpha^2 t^2}.$$
As a consequence, for every $2\times 2$ matrices $X$ and $Y$, 
\begin{equation}\label{EQ_estimation_crochets}
\lim_{t\to 0} \frac{1}{t^2}\left \| e^{t\alpha X}e^{t\alpha Y}e^{-t \alpha X}e^{-t \alpha Y} - e^{\alpha ^2 t^2[X,Y]}\right \| = 0,
\end{equation}
 the convergence being uniform with respect to $\alpha$ in $[0,1]$.

Recall that, for every $V,W$ in $\mathfrak{su}(2)$, for every $n$ in $\mathbf{N}$, 
\begin{eqnarray}
\|V^n-W^n\|&=& V(V^{n-1}-W^{n-1})+(V-W)W^{n-1} \nonumber\\
&\leq & \|V^{n-1}-W^{n-1}\|+\|V-W\| \nonumber\\
&\leq & n\|V-W\|.\label{EQ_maj_diff_puissance}
\end{eqnarray}
Hence, 
\begin{eqnarray}
\lefteqn{\left \| \left ( e^{t\alpha X}e^{t\alpha Y}e^{-t \alpha X}e^{-t \alpha Y}
\right )^n- e^{n \alpha^2 t^2[X,Y]}\right \| \leq } \nonumber\\
&\quad \quad &n \left \| e^{t\alpha X}e^{t\alpha Y}e^{-t \alpha X}e^{-t \alpha Y} - e^{\alpha ^2 t^2[X,Y]}\right \|\label{EQ_maj_crochets_ordre_1}
\end{eqnarray}
Choosing $n=T/t^2$ and letting $n$ tend to infinity (and hence $t$ tend to zero) gives the result for $p=1$, 
$\beta_1=1$ and $l_1=1$. The proof for $l_1>1$ is very similar, replacing $X$ and $Y$ by 
the suitable iterated brackets in (\ref{EQ_maj_crochets_ordre_1}).

A consequence of Zassenhauss formula is that, for every $2\times 2$ matrices $U,V$, there exists a locally Lipschitz function $g:\mathfrak{gl}(2)\times\mathfrak{gl}(2)\times \mathbf{R}\to \mathbf{R}$ that vanishes as soon as one of its entries vanishes  such that, for every $t$ in $\mathbf{R}$, for every $\alpha$ in $[0,1]$, for every $j,k$ in $\mathbf{N}$, 
$$ \|e^{t  U}e^{t  V}-e^{t  (U + V)}\|\leq t^2 g(U,V,t).$$
The proof of Lemma \ref{LEM_suivi_polynomes_crochets}, for $p>1$ and $\beta_j$ not necessarily equal to 1, follows by choosing $t=T/n$ for $n$ large enough and using once again (\ref{EQ_maj_diff_puissance}). 
\end{proof}

We come back to the proof of Proposition \ref{PRO_contr_syst_aux}.
After the time dependent change of variable
\begin{equation}\label{EQ_chgt_de_variable}
y_\alpha=\exp \left \lbrack -t \alpha  \left ( 
\begin{array}{cc}
b_{11} & 0\\ 0 & b_{22}
\end{array}
\right )\right \rbrack x_\alpha,
\end{equation}
the system (\ref{EQ_syst_aux_avat_chgt_variable}) reads
\begin{equation}
y_{\alpha}'= \alpha \left ( \begin{array}{ll} 0 & b_{12}e^{\mathrm{i}\theta-t(b_{11}-b_{22})}\\
{b_{21}}e^{-\mathrm{i}\theta+t(b_{11}-b_{22})} & 0 \end{array} \right ) y_{\alpha},
\end{equation}
or $y_\alpha'=M^\nu_\alpha y_\alpha,$ defining $\nu= \theta -\mathrm{i}t (b_{11}-b_{22})$,
with
\begin{equation}\label{EQ_Vitesse_elementaire_syst_aux}
M^\nu_\alpha=\alpha \left ( \begin{array}{ll} 0 & b_{12}e^{\mathrm{i}\nu}\\
{b_{21}}e^{-\mathrm{i}\nu} &0 \end{array} \right ).
\end{equation}

\begin{lemma}\label{PRO_egalite_modules}
For every $\phi$ in $\mathbf{C}^2$, for every $t$ in $\mathbf{R}$, for every 
locally integrable $\theta:\mathbf{R}\to \mathbf{R}$, the moduli of the coordinates in the 
canonical basis  $(\phi_1,\phi_2)$ of $\mathbf{C}^2$ of $y_\alpha(t)\psi$ and 
$x_\alpha(t)\psi$ are the same. 
\end{lemma}
\begin{proof}
From (\ref{EQ_chgt_de_variable}), the coordinates of $y_\alpha$ and $x_\alpha$ are equal, 
up to a phase shift  depending on time and  $\alpha$. 
\end{proof}
 Thanks to Lemma \ref{LEM_suivi_polynomes_crochets},  Proposition \ref{PRO_contr_syst_aux} 
 follows if, for every $l$, the matrix $Z_{2l+1}$ defined above can be realized as a 
 linear combination (with real coeeficients) of brackets of length exactly equal to $2l+1$ 
 of the matrices $M^\nu_1, \nu \in [-\pi,\pi]$. 
Notice that 
$$
M^0_1=  \left (\begin{array}{ll} 0 & b_{12} \\ b_{21}& 0\end{array} \right ) 
\mbox{ and } M^{\frac{\pi}{2}}_1= 1 \left (\begin{array}{ll} 0 & \mathrm{i} b_{12} \\ 
-\mathrm{i}b_{21}& 0\end{array} \right ).
$$
Straightforward computations give, for every $k$ in $\mathbf{N}$,
\begin{equation}\label{EQ_crochets_iteres}
ad_{M^{\frac{\pi}{2}}_\alpha}^k M^0_\alpha= \alpha^{2k+1} \left ( \begin{array}{ll} 0 & b_{12} |b_{12}|^{2k}\\ b_{21} |b_{12}|^{2k} &0  \end{array}\right ),
\end{equation}
Proposition \ref{PRO_contr_syst_aux} follows from the fact that $b_{12}\neq 0$. 
\end{proof}

\subsection{Averaging techniques}\label{SEC_averaging}

We define the $N\times N$ matrix $N^\theta_{\alpha}$ by $N^\theta_{\alpha}(j,k)=0$ for 
every $j,k$ in $\{1,\ldots,N\}^2$ but  
$N^\theta_{\alpha}(1,2)=\alpha b_{12}e^{\mathrm{i}\theta}$ and 
$N^\theta_{\alpha}(2,,1)=-\overline{N^\theta_{\alpha}(1,2)}$. 
In particular, $N^\theta_{\alpha}$ belongs to $\mathfrak{su}(N)$ and 
$\pi_2 N^\theta_{\alpha}\pi_2=M^\theta_\alpha$. 

Let us come back to the proof of Proposition \ref{PRO_main_result_dim_fin}. 
From Proposition \ref{PRO_contr_syst_aux}, it is enough to show that, for every $\theta,t$ 
in $\mathbf{R}$ and every $\varepsilon>0$, there exists 
$u_{\varepsilon}:[0,T_\varepsilon]\to (-\delta,\delta)$ such that, for every 
$\alpha$ in $[0,1]$, 
$\|\pi_2 X^{u_\varepsilon}_{(N)}(T_\varepsilon)-e^{t N^\theta_{\alpha}}\|<\varepsilon$.
This is exactly the content of Proposition \ref{PRO_estimates_dim_finie}, whose proof is given in \cite{periodic}.
\begin{proposition}\label{PRO_estimates_dim_finie}
Let $u^{\ast}:\mathbf{R}^+\rightarrow \mathbf{R}$ be a locally integrable function.

 Assume that $u^{\ast}$ is periodic with period $T=\frac{2\pi}{|\lambda_2-\lambda_1|}$
 and that $\displaystyle{\int_0^T \!\!\! u^{\ast} (\tau) e^{  \mathrm{i}
 (\lambda_{l}-\lambda_{m})\tau} \mathrm{d}\tau = 0}$
for every $\{l,m\}$ such that $\{l,m\}\cap \{1,2\}\neq \emptyset$ and ${\lambda_l-\lambda_m}\in (\mathbf{Z}\setminus\{\pm1\})(\lambda_1-\lambda_2)$ and
$b_{lm} \neq 0$.
  For every $n$, define $v_n:t\mapsto 1/n \int_0^t|u^{\ast}(s)|\mathrm{d}s$ and the
  $N\times N$ matrix $M^\dag$ with entries $m_{j,k}^{\dag}=b_{jk} {\int_0^T u^\ast(s)
   e^{\mathrm{i}(\lambda_2-\lambda_1)s} \mathrm{d}s}/{\int_0^T |u^\ast(s)|\mathrm{d}s}$.

If $\displaystyle{\int_0^T \!\!\! u^{\ast} (\tau) e^{  \mathrm{i}(\lambda_{2}-\lambda_{1})\tau} \mathrm{d}\tau \neq 0}$,
 then, for every $n$ in $\mathbf{N}$, for every $t\leq nT^\ast$, 
\begin{eqnarray}\label{EQ_major_dist_propagateur}
\frac{\|X^{u_n}_{(N)}(t,0)-e^{tA^{(N)}}e^{v_n^{[-1]}(t) M^{\dag}}\|}{I(C+1)  \|B^{(N)}\|} 
\leq
\frac{ 1 + 2K \|B^{(N)}\|}{n}\label{EQ_estimation_conv_uniform_propaga_dim_finie}.
\end{eqnarray}
with 
$$T^{\ast}=  \frac{\pi  T}{2 |b_{1,2}|  \left |\int_0^T \!\! \!u^{\ast}(\tau)e^{\mathrm{i}
(\lambda_{1}-\lambda_{2}) \tau}  \mathrm{d}\tau \right |},\quad I=\int_0^T \! \!|u^{\ast}
(\tau)|\mathrm{d}\tau,$$
$$ K=\frac{IT^{\ast}}{T}, C=\sup_{(j,k)\in \Lambda} \left | \frac{\int_0^T u^{\ast}(\tau) 
e^{\mathrm{i} (\lambda_j-\lambda_k)\tau}\mathrm{d}\tau}{\sin \left ( \pi\frac{|
\lambda_j-\lambda_k|}{|\lambda_2-\lambda_1|} \right )} \right |,$$
where $\Lambda$ is the set of all pairs $(j,k)$ in  $\{1,\ldots,N\}^2$  such that $b_{jk} 
\neq 0$ and $\{j,k\}\cap\{1,2\} \neq \emptyset$ and $ |\lambda_j-\lambda_k|\notin 
\mathbf{Z}|\lambda_2-\lambda_1|$.
\end{proposition} 

\begin{proof}(Proposition \ref{PRO_main_result_dim_fin})
 We apply Proposition \ref{PRO_estimates_dim_finie} with $u^\ast$  
periodic with period $T=\frac{2\pi}{|\lambda_2-\lambda_1|}$
 and satisfying  $\displaystyle{\int_0^T \!\!\! u^{\ast} (\tau) e^{  \mathrm{i}(\lambda_{2}-\lambda_{1})\tau} \mathrm{d}\tau \neq 0}$ $\displaystyle{\int_0^T \!\!\! u^{\ast} (\tau) e^{  \mathrm{i}
 (\lambda_{l}-\lambda_{m})\tau} \mathrm{d}\tau = 0}$
for every $\{l,m\}$ such that $\{l,m\}\cap \{1,2\}\neq \emptyset$ and ${\lambda_l-\lambda_m}\in (\mathbf{Z}\setminus\{\pm1\})(\lambda_1-\lambda_2)$ and
$b_{lm} \neq 0$. Such a $u^\ast$ can be chosen of the form $t\mapsto \cos((\lambda_2-\lambda_1)t-\theta)$ or piecewise constant with value in $\{0,1\}$ (for an explicit construction of such a function, see \cite{Schrod2}). 

To ensure that $e^{v_n^{[-1]}(t)M^\dag}$ is $\varepsilon$-close to $e^{ t M_\alpha^\nu}$,
 defined as  in (\ref{EQ_Vitesse_elementaire_syst_aux}), one chooses $t$ such that 
 $v_n(t)=b_{12}e^{\mathrm{i}\theta}/m_{12}^\dag \leq n T^\ast$. One can check from the 
 definition of $v_n$ that $t/(n r^\ast) $ tends to 1 as $n$ tends to infinity, where $r^\ast$ is defined by  
 $$r^\ast= \frac{T}{I} v_1^{[-1]}\left (\frac{b_{12}e^{\mathrm{i}\theta}}
 {m_{12}^\dag}\right )  .$$ 

The final step in the proof of Proposition \ref{PRO_main_result_dim_fin} is to get rid of 
the phase $e^{tA^{(N)}}$ in estimate (\ref{EQ_estimation_conv_uniform_propaga_dim_finie}).
We use the Poincar\'e Recurrence Theorem with the mapping $$
\begin{array}{lcl}
\mathbf{R} & \to & \mathbf{R}^{n+1}/\mathbf{Z}^ {n+1}\\
s&\mapsto &\left ( 
\frac{\lambda_1}{2\pi} s,\frac{\lambda_2}{2\pi}  s,\ldots,\frac{\lambda_n}{2\pi} s, 
\frac{1}{r^\ast} s \right )
\end{array}$$ on the $n+1$ dimensional torus. For every $\varepsilon>0$, there exists a 
sequence $(s_k)_{k\in \mathbf{N}}$ that tends to infinity such that $s_k \lambda_j$ is 
$\varepsilon$ close to $2\pi \mathbf{Z}$ and $s_k$ is $\varepsilon$-close to $ r^\ast
\mathbf{Z}$. The sequence of controls $\lfloor r^\ast/s_k \rfloor u^{\ast} $ gives Proposition 
\ref{PRO_main_result_dim_fin} by letting $k$ tend to infinity.
\end{proof}

\section{INFINITE DIMENSIONAL ESTIMATES}
\subsection{Heuristic of the proof}
In this Section, we proceed to the proof of Proposition \ref{PRO_main_result}. 
Inspired by Section \ref{SEC_Finite_dimension}, it is enough to show that the projections 
of each of the infinite dimensional systems (\ref{EQ_bilinear_continuum}) can track, with 
an arbitrary precision, the trajectories of the $2\times 2$ system 
(\ref{EQ_syst_aux_avat_chgt_variable}).

To begin with, we consider  in Section \ref{SEC_spectre_discret} a system that satisfies 
Assumption \ref{ASS_AB_spectre_discret}. The proof is  a uniform version of the 
Section 4 of \cite{periodic} which is valid for one particular $\alpha$. 

To prove Proposition \ref{PRO_main_result} for systems that satisfy 
Assumption \ref{ASS_AB}, we first estimate the robustness of the results of 
Section \ref{SEC_spectre_discret} against a perturbation of the spectrum of $A$. 
The conclusion will follow from the Von Neumann approximation theorem. As in Section 
\ref{SEC_spectre_discret}, the method of the proof in Section \ref{SEC_spectre_continu} is 
similar to the one used in \cite{periodic}, the only difference lying once again in the 
uniformity of the convergence estimates with respect to $\alpha$ in $[0,1]$.

\subsection{If the eigenvectors of $A$ span a dense subspace of $H$}\label{SEC_spectre_discret}

Let $(A,B,U,\Lambda,\Phi)$ satisfy Assumption \ref{ASS_AB_spectre_discret}, $\theta$ in 
$[-\pi,\pi]$ and $r,\varepsilon > 0$.
We aim to find a periodic control $u^\ast$ with period $2\pi/(\lambda_2-\lambda_1)$, $n$ 
in $\mathbf{N}$ and $T$ in $\mathbf{R}^+$ such that, for every $\alpha$ in $[0,1]$,
$$\|\Upsilon^{u^\ast/n,\alpha}_{T,0} - e^{r M^\theta_\alpha}\|<\varepsilon.$$

Since $\phi_1$ and $\phi_2$ belong to the domain of $B$, the sequences $(b_{1,l})_{l\in \mathbf{N}}$ and
$(b_{2,l})_{l\in \mathbf{N}}$ are  in $\ell^2$. Hence, there exists $N$ in $\mathbf{N}$ such that $\|\pi_2 B(1-\pi_N)\|=\|(1-\pi_N) B\pi_2 \|<5  \varepsilon/(2r) .$
Define $\omega=\lambda_2-\lambda_1$ and $u^\ast$ with period $2\pi/\omega$ in such a way 
that the $N\times N$ matrix $M^\dag$ of Proposition \ref{PRO_estimates_dim_finie} is equal 
to $M^\theta_1$ and the efficiency
 $|\int_0^{\frac{2\pi}{\omega}} u^\ast(s) e^{\mathrm{i}\omega s}\mathrm{d}s|
 /\int_0^{\frac{2\pi}{\omega}} |u^\ast(s)| \mathrm{d}s  $ of $u^\ast$ for the transition 
 $(1,2)$ is larger than $2/5$. This can be done, for instance, with 
 $t\mapsto \cos(wt-\theta)$ (efficiency $\pi/4$) in the case where $B$ is bounded or, in 
 the general case of Assumption \ref{ASS_AB_spectre_discret}, with a piecewise constant 
 function taking value in $\{0,1\}$ as described in \cite{Schrod2}.

For a given $n$ to be precised later, consider system (\ref{EQ_main}) with control $u_n=u^\ast/n$ in projection   on $\mathrm{span}(\phi_1,\ldots,\phi_N)$:
\begin{eqnarray}
 \pi_N \frac{\mathrm{d}}{\mathrm{d}t} \Upsilon^{u_n,\alpha}_t \phi_j &=&
(A^{(N)} + u_n(t)\alpha  B^{(N)}) \Upsilon^{u_n,\alpha}_t \phi_j \nonumber \\&&+u_n(t) \pi_N \alpha B (1-\pi_N) \Upsilon^{u_n,\alpha}_t \phi_j.~~~~
\end{eqnarray}
From the variation of the constant, we get, for $j=1,2$,
\begin{eqnarray}
\lefteqn{\pi_N \Upsilon^{u_n,\alpha}_t \phi_j= X_{(N)}^{u_n,\alpha}(t,0) \phi_j} \nonumber\\
&\! &+\!\!\! \int_0^t\!\!\!u_n(s) X^{u_n,\alpha}_{(N)}(t,s) \pi_N \alpha B (1-\pi_N) \Upsilon^{u_n,\alpha}_t \phi_j \mathrm{d}s.~~~\label{EQ_proj_N}
\end{eqnarray}
Project (\ref{EQ_proj_N}) on $\mathrm{span}(\phi_1,\phi_2)$, and recall that $\pi_N \pi_2=\pi_2 \pi_N=\pi_2$ for $N\geq 2$:
\begin{eqnarray}
 \lefteqn{\pi_2 \Upsilon^{u_n,\alpha}_t \phi_j= \pi_2 X_{(N)}^{u_n,\alpha}(t,0) \phi_j} 
  \label{EQ_pre_bracket}\\&&+ \int_0^t \!\!\!u_n(s) \pi_2 X_{(N)}^{u_n,\alpha}(t,s) \pi_N \alpha B (1-\pi_N) \Upsilon^{u_n,\alpha}_t \phi_j \mathrm{d}s.\nonumber
\end{eqnarray}
Define, for every $t,s$ in $\mathbf{R}$, the bounded linear mapping $[\pi_2, X^{u_n,\alpha}_{(N)}(t,s)]:=\pi_2\circ X_{(N)}^{u_n,\alpha}(t,s)-X_{(N)}^{u_n,\alpha}(t,s) \circ \pi_2$. Equation (\ref{EQ_pre_bracket}) reads, for $j=1,2$,
\begin{eqnarray}
\lefteqn{\pi_2 \Upsilon^{u_n,\alpha}_t \phi_j- \pi_2 X_{(N)}^{u_n,\alpha}(t,0) \phi_j =}\nonumber \\
&&
-{\int_0^t\!\!\! u_n(s)  X_{(N)}^{u_n,\alpha}(t,s) \pi_2 \alpha B (1-\pi_N) \Upsilon^{u_n,\alpha}_t \phi_j \mathrm{d}s}  \label{EQ_post_bracket} \\
&& +
{\int_0^t\!\!\! u_n(s) [\pi_2, X_{(N)}^{u_n,\alpha}(t,s)]  \pi_N \alpha B (1-\pi_N) \Upsilon^{u_n,\alpha}_t \phi_j \mathrm{d}s}.\nonumber
\end{eqnarray}
Extend the definition of $M^{\alpha}_\theta$ to $H$ by $M^{\dag}=0$ on ${\mathcal L}_N^\perp$ and
define the linear operator $E^{n,\alpha}_{N}(t):=X_{(N)}^{u_n}(t,0)-e^{v_n^{[-1]}(t)M^{\alpha}_\theta}\!\!$.
 Since the commutator $[\pi_2,M^{\alpha}_\theta]=\pi_2 M^{\alpha}_\theta-M^{\alpha}_\theta \pi_2$ vanishes, we have, for every $t$ in $\mathbf{R}$,
\begin{eqnarray*}
\|[\pi_2, X_{(N)}^{u_n,\alpha}(t,0)]\|&=&\|[\pi_2, e^{v^{[-1]}(t)M^{\dag}} +E^{n,\alpha}_{(N)}(t)]\|\\
&=&\|[\pi_2,E^{n,\alpha}_{(N)}(t)]\|\leq 2 \|E^{n,\alpha}_{(N)}(t)\|.
\end{eqnarray*}
Note also that, for every $t$ in $\mathbf{R}$,
\begin{eqnarray*}
\lefteqn{\|[\pi_2, X_{(N)}^{u_n}(0,t)]\|\!\!=\!\!
\|X_{(N)}^{u_n}\!(0,t) [X_{(N)}^{u_n}\!(t,0), \pi_2] X_{(N)}^{u_n}\!(0,t)    \|}\\
&\quad \quad \quad \quad \quad  \quad \,  \leq &2 \|E^{n,\alpha}_{(N)}(t)\|.\quad \quad \quad \quad\quad \quad \quad \quad 
\end{eqnarray*}
For every $s,t$ in $\mathbf{R}$,
\begin{eqnarray*}
 \lefteqn{[\pi_2, X_{(N)}^{u_n}(t,s)]}\\
&=&\pi_ 2X_{(N)}^{u_n}(t,0)X_{(N)}^{u_n}(0,s)-X_{(N)}^{u_n}(t,0)X_{(N)}^{u_n}(0,s) \pi_2\\
&=&X_{(N)}^{u_n}\!(t,\!0)[\pi_ 2,\! X_{(N)}^{u_n}(0,s)]\!+\! [\pi_2,\! X_{(N)}^{u_n} (t,0)]X_{(N)}^{u_n}(0,\!s).
\end{eqnarray*}
Finally, we get, for every $(s,t)$ in $\mathbf{R}^2$, for every $n,N$ in $\mathbf{N}$.
\begin{equation}\label{EQ_maj_commutateurs}
 \left \|\left [ \pi_2, X_{(N)}^{u_n}(t,s) \right \rbrack \right \|\leq  4 \|E_{(N)}^{n,\alpha}(t)\|.
\end{equation}

From (\ref{EQ_post_bracket}) and (\ref{EQ_maj_commutateurs}), since $\|\pi_2B(1-\pi_N)\|<\varepsilon/K$,
\begin{eqnarray}
 \lefteqn{ \|\pi_2 \Upsilon^{\frac{u^{\ast}}{n}}(t)\pi_2  -\pi_2 X^{\frac{u^{\ast}}{n}}_{(N)}(t,0)\pi_2 \|} \nonumber\\ &\quad \quad \quad \leq&
\varepsilon  + 4 \|E_{(N)}^{n,\alpha}(t)\| K \|\pi_N B(1-\pi_N)\|.\label{EQ_majoration_generale}
\end{eqnarray}
From (\ref{EQ_estimation_conv_uniform_propaga_dim_finie}), $\sup_{t\leq v_n(K)} \|E^{n,\alpha}_{(N)}(t)\|$ tends to zero as $n$ tends to infinity.
For $n$ large enough, for every $\alpha$ in $[0,1]$ and $t\leq v_n(r)$, $$\|E_{(N)}^{n,\alpha}(t)\|\leq  \frac{\varepsilon}{4K \|\pi_N B(1-\pi_N)\|}.$$ 
 Proposition \ref{PRO_estimates_dim_finie} completes the proof of Proposition \ref{PRO_main_result} in the case where $(A,B,U,\Lambda,\Phi)$ satisfies Assumption \ref{ASS_AB_spectre_discret}.

\subsection{If $A$ has a mixed spectrum}\label{SEC_spectre_continu}
Assume that $(A,B,\lambda_1,\phi_1,\lambda_2,\phi_2)$ satisfy Assumption \ref{ASS_AB}.
From Theorem 2.1, page 525, of \cite{kato}, for every $\eta>0$, there exists a 
skew-adjoint operator $A_{\eta}$ such that $A_{\eta}$ admits a complete family of 
eigenvectors $(\Phi_\eta)$ associated with the family of eigenvalues
 $(\Lambda_\eta)$, $A \phi=A_{\eta}\phi$  for every eigenvector $\phi$ of $A$
 and $\|A-A_{\eta}\|<\eta$.

For every locally integrable $u$, we denote with $\Upsilon^{u}_\eta$ the propagator of $\frac{d}{dt}\psi=(A_{\eta}+uB_M)\psi$.

The scheme of the proof is a follows: the result is known (from Section 
\ref{SEC_spectre_discret}) for the system $(A_{\eta},B,\mathbf{R},\Lambda_\eta,
\Phi_\eta)$, 
which satisfies Assumption \ref{ASS_AB_spectre_discret}: we chose 
$u^\ast:t\mapsto \cos(\omega t -\theta)$ 
(this function is the ``shape'' of the control pulses, it does not depend on 
$\eta$ nor $\varepsilon$). For every $\eta,\varepsilon>0$, 
$\theta$ in $[-\pi,\pi]$ and $r>0$,  there exists an 
integer $n_\eta$ and a positive real $T_\eta$ satisfying, for every $\alpha$ in $[0,1]$, 
$$\|\Upsilon^{u^\ast/n,\alpha,\eta}_{T_\eta,0} - e^{r M^\theta_\alpha}\|<\varepsilon.$$
Notice that, for every $t$ in $\mathbf{R}$, 
$\|\Upsilon^{u^\ast/n,\alpha,\eta}_{t,0}-\Upsilon^{u^\ast/n,\alpha}_{t,0}\| \leq |t|\eta$.
The crucial point in the proof of Proposition \ref{PRO_main_result} for systems satisfying 
Assumption \ref{ASS_AB} is the existence of a uniform bound on $T_\eta$, that depends only 
on $r$ and $\varepsilon$, and not on $\eta$. This follows from 
(\ref{EQ_estimation_conv_uniform_propaga_dim_finie}), where the only variable 
depending on $\eta$ is 
$$C=\sup_{(j,k)\in \widehat{\Lambda}} \left | \frac{\int_0^T u^{\ast}(\tau) 
e^{\mathrm{i} (\lambda_j-\lambda_k)\tau}\mathrm{d}\tau}{\sin \left ( \pi\frac{|
\lambda_j-\lambda_k|}{|\lambda_2-\lambda_1|} \right )} \right |,$$
where $\widehat{\Lambda}$ is the set of all pairs $(j,k)$ in  $\{1,\ldots,N\}^2$  such that $b_{jk} 
\neq 0$ and $\{j,k\}\cap\{1,2\} \neq \emptyset$ and $ |\lambda_j-\lambda_k|\notin 
\mathbf{Z}|\lambda_2-\lambda_1|$.
(Notice that $\|B^{(N)}\|$ is bounded, for every $N$ by $\|B\|$.)

Straightforward computation gives $C \leq 2/d$ where $d$ is the distance of the set $\{2\lambda_1-\lambda_2,\lambda_1,\lambda_2,2\lambda_2-\lambda_1\}$ to the continuous part of the spectrum of $\mathrm{i}A$. This distance is not zero by Assumption 
\ref{ASS_AB}.\ref{ASS_sspectre_essentiel}, what concludes the proof of Proposition \ref{PRO_main_result}.


\addtolength{\textheight}{-0.8cm}

\section{EXAMPLE: ROTATION OF A MOLECULE}\label{SEC_Example}
\subsection{Modeling}
The description of the physical system we consider is a toy model inspired by the physical system described in \cite{6426289}. It has already been thoroughly studied, see for instance \cite{noiesugny-CDC}, \cite{Schrod2} or \cite{ACCFEPS}).
We consider a polar linear molecule in its ground vibronic state subject to a nonresonant (with respect to the vibronic frequencies) linearly polarized laser field. 
The control is given by the electric field $E=u(t)(E_1,E_2,E_3)$ depending on time and constant in space. We neglect in this model the polarizability tensor term which corresponds to the field-induced dipole moment (see for instance \cite{Morancey} or \cite{CDCquadratic}). 

Let $\mathcal{P}$ be a fixed plane in the space. We are interested in the orientation of 
the orthogonal projection of a set of molecules in the plane $\mathcal{P}$ (given by \emph{one} 
angle, in contrary to the orientation of the molecule in the space which is given by two 
angles). We neglect the interaction between molecules, and consider only the interaction between the molecules and the external field.
Our aim is to control the orientation of projection of the molecule in $\mathcal P$, 
whatever the angle of the molecule could be with $\cal P$. 

Up to normalization of physical constants (in particular, in units such that $\hbar=1$), 
the dynamics of each molecule is ruled by the equation
\begin{equation}
\mathrm{i}\frac{\partial\psi(\theta,t)}{\partial t}= -\Delta \psi + u_1(t) \cos 
\theta\sin \varphi \psi(\theta,\varphi,t)
\label{EQ_dyn_rotation}
\end{equation}
where $\theta$ is the angular coordinate in $\mathcal P$ and $\varphi$ is the angle of the 
molecule with $\mathcal P$, which is assumed to be constant for the sake of simplicity, 
$\Delta$ is the Laplace--Beltrami operator on the circle 
$\mathbf{S}=\mathbf{R}/2\pi\mathbf{Z}$, 
The wavefunction $\psi(\cdot,t)$ evolves in the unit sphere ${\cal S}$ of 
$H=L^2(\mathbf{S},\mathbf{C})$ endowed with scalar product $\langle 
f,g\rangle=\int_0^{2\pi}\bar{f}(s)g(s)\mathrm{d}s$.

The operator $A=\mathrm{i}\Delta$ is skew-adjoint in $H$, with domain $H^2(\mathbf{S},
\mathbf{C})$ and has discrete spectrum. Define $\phi_0:\theta \mapsto 1/\sqrt{2\pi}$ and, for every $k$ in 
$\mathbf{N}$, $\phi_{2k-1}:\theta\mapsto \cos(k \theta)/\sqrt{\pi}$ and $\phi_{2k}:\theta\mapsto \sin(k \theta)/\sqrt{\pi}$. 
The two functions $\phi_{2k+1}$ and $\phi_{2k}$ are eigenvectors of $A$, associated with 
eigenvalue $-\mathrm{i} k^2$.

The operator $B:\psi \mapsto -\mathrm{i}\cos(\theta) \psi $ is bounded. Straightforward
computations show that $|\langle \phi_0,B,\phi_1\rangle|=1/\sqrt{2}$ and 
$\langle \phi_j,\phi_k \rangle = 0$ if the parities of $j$ and $k$ are different 
or if $|j-k|>2$.

\subsection{Result}

Assume that a bunch of molecules is in the state $\phi_0$ at $t=0$. We aim to transfer 
to the state $\phi_1$ all the molecules for which $\varphi>\pi/3$ and to keep all the 
molecules for which $\varphi<\pi/6$ in the state $\phi_0$. 

From Proposition \ref{PRO_main_result}, applied to $(A,B,0,\phi_0,1,\phi_2)$ which satisfies Assumption \ref{ASS_AB}, this is possible, up to the phase and to an arbitrary small error $\varepsilon>0$.

\section{CONCLUSIONS}

\subsection{Some comments on the result}
While our construction is completely explicit (simple formulas are available for the control laws 
and come along with precision and time estimates), the convergence toward the target is 
extremely slow and cannot be used for actual control of real systems. This well-known fact 
is due to the very poor efficiency of tracking strategies via Lie brackets. 

\subsection{Perspectives}
The presented results may certainly be improved in many ways. For instance, the author 
conjectures that it is possible to replace $\hat{\Upsilon}$ in 
Proposition \ref{PRO_main_result} by a unitary 
transformation of $\mathcal{L}_N$ with $N>2$ or to extend the result to systems for which the free Hamiltonian $A$ has a mixed spectrum and the coupling Hamiltonian $B$ is unbounded. 


\section{ACKNOWLEDGMENTS}

This work has been partially supported by INRIA Nancy-Grand Est, by
French Agence National de la Recherche ANR ``GCM'' program
``BLANC-CSD'', contract number NT09-504590 and by European Research
Council ERC StG 2009 ``Ge\-Co\-Methods'', contract number
239\-748.

\bibliographystyle{IEEEtran}
\bibliography{biblio}

\begin{thebibliography}{10}
\providecommand{\url}[1]{#1}
\csname url@samestyle\endcsname
\providecommand{\newblock}{\relax}
\providecommand{\bibinfo}[2]{#2}
\providecommand{\BIBentrySTDinterwordspacing}{\spaceskip=0pt\relax}
\providecommand{\BIBentryALTinterwordstretchfactor}{4}
\providecommand{\BIBentryALTinterwordspacing}{\spaceskip=\fontdimen2\font plus
\BIBentryALTinterwordstretchfactor\fontdimen3\font minus
  \fontdimen4\font\relax}
\providecommand{\BIBforeignlanguage}[2]{{%
\expandafter\ifx\csname l@#1\endcsname\relax
\typeout{** WARNING: IEEEtran.bst: No hyphenation pattern has been}%
\typeout{** loaded for the language `#1'. Using the pattern for}%
\typeout{** the default language instead.}%
\else
\language=\csname l@#1\endcsname
\fi
#2}}
\providecommand{\BIBdecl}{\relax}
\BIBdecl

\bibitem{bms}
J.~M. Ball, J.~E. Marsden, and M.~Slemrod, ``Controllability for distributed
  bilinear systems,'' \emph{SIAM J. Control Optim.}, vol.~20, no.~4, pp.
  575--597, 1982.

\bibitem{Turinici}
G.~Turinici, ``On the controllability of bilinear quantum systems,'' in
  \emph{Mathematical models and methods for ab initio Quantum Chemistry}, ser.
  Lecture Notes in Chemistry, M.~Defranceschi and C.~{Le Bris}, Eds.,
  vol.~74.\hskip 1em plus 0.5em minus 0.4em\relax Springer, 2000.

\bibitem{beauchard}
K.~Beauchard, ``Local controllability of a 1-{D} {S}chr{\"o}dinger equation,''
  \emph{J. Math. Pures Appl.}, vol.~84, no.~7, pp. 851--956, 2005.

\bibitem{beauchard-coron}
K.~Beauchard and J.-M. Coron, ``Controllability of a quantum particle in a
  moving potential well,'' \emph{J. Funct. Anal.}, vol. 232, no.~2, pp.
  328--389, 2006.

\bibitem{nersesyan}
V.~Nersesyan, ``Global approximate controllability for {S}chr{\"o}dinger
  equation in higher {S}obolev norms and applications,'' \emph{Ann. Inst. H.
  Poincar{\'e} Anal. Non Lin{\'e}aire}, vol.~27, no.~3, pp. 901--915, 2010.

\bibitem{Nersy}
------, ``Growth of {S}obolev norms and controllability of the
  {S}chr{\"o}dinger equation,'' \emph{Comm. Math. Phys.}, vol. 290, no.~1, pp.
  371--387, 2009.

\bibitem{beauchard-nersesyan}
K.~Beauchard and V.~Nersesyan, ``Semi-global weak stabilization of bilinear
  {S}chr\"odinger equations,'' \emph{C. R. Math. Acad. Sci. Paris}, vol. 348,
  no. 19-20, pp. 1073--1078, 2010.

\bibitem{Mirrahimi}
K.~Beauchard, J.~M. Coron, M.~Mirrahimi, and P.~Rouchon, ``Implicit {L}yapunov
  control of finite dimensional {S}chr{\"o}dinger equations,'' \emph{Systems
  Control Lett.}, vol.~56, no.~5, pp. 388--395, 2007.

\bibitem{MR2168664}
M.~Mirrahimi, P.~Rouchon, and G.~Turinici, ``Lyapunov control of bilinear
  {S}chr{\"o}dinger equations,'' \emph{Automatica J. IFAC}, vol.~41, no.~11,
  pp. 1987--1994, 2005.

\bibitem{mirra-solo}
M.~Mirrahimi, ``Lyapunov control of a particle in a finite quantum potential
  well,'' in \emph{Proceedings of the 45th IEEE Conference on Decision and
  Control}, December 2006.

\bibitem{Schrod}
T.~Chambrion, P.~Mason, M.~Sigalotti, and U.~Boscain, ``Controllability of the
  discrete-spectrum {S}chr{\"o}dinger equation driven by an external field,''
  \emph{Ann. Inst. H. Poincar{\'e} Anal. Non Lin{\'e}aire}, vol.~26, no.~1, pp.
  329--349, 2009.

\bibitem{Schrod2}
U.~Boscain, M.~Caponigro, T.~Chambrion, and M.~Sigalotti, ``A weak spectral
  condition for the controllability of the bilinear {S}chr\"odinger equation
  with application to the control of a rotating planar molecule,'' \emph{Comm.
  Math. Phys.}, vol. 311, no.~2, pp. 423--455, 2012.

\bibitem{PhysRevA.73.030302}
\BIBentryALTinterwordspacing
J.-S. Li and N.~Khaneja, ``Control of inhomogeneous quantum ensembles,''
  \emph{Phys. Rev. A}, vol.~73, p. 030302, Mar 2006. [Online]. Available:
  \url{http://link.aps.org/doi/10.1103/PhysRevA.73.030302}
\BIBentrySTDinterwordspacing

\bibitem{MR2879412}
\BIBentryALTinterwordspacing
K.~Beauchard, P.~S. Pereira~da Silva, and P.~Rouchon, ``Stabilization for an
  ensemble of half-spin systems,'' \emph{Automatica J. IFAC}, vol.~48, no.~1,
  pp. 68--76, 2012. [Online]. Available:
  \url{http://dx.doi.org/10.1016/j.automatica.2011.09.050}
\BIBentrySTDinterwordspacing

\bibitem{MR2608124}
\BIBentryALTinterwordspacing
K.~Beauchard, J.-M. Coron, and P.~Rouchon, ``Controllability issues for
  continuous-spectrum systems and ensemble controllability of {B}loch
  equations,'' \emph{Comm. Math. Phys.}, vol. 296, no.~2, pp. 525--557, 2010.
  [Online]. Available: \url{http://dx.doi.org/10.1007/s00220-010-1008-9}
\BIBentrySTDinterwordspacing

\bibitem{4177466}
J.-S. Li and N.~Khaneja, ``Ensemble controllability of the bloch equations,''
  in \emph{Decision and Control, 2006 45th IEEE Conference on}, dec. 2006, pp.
  2483 --2487.

\bibitem{MR2191545}
\BIBentryALTinterwordspacing
------, ``Ensemble control of {B}loch equations,'' \emph{IEEE Trans. Automat.
  Control}, vol.~54, no.~3, pp. 528--536, 2009. [Online]. Available:
  \url{http://dx.doi.org/10.1109/TAC.2009.2012983}
\BIBentrySTDinterwordspacing

\bibitem{SarletteEnsemble}
Z.~Leghtas, A.~Sarlette, and P.~Rouchon, ``Adiabatic passage and ensemble
  control of quantum systems,'' \emph{Journal of Physics B}, vol.~44, p.
  154017, 2011.

\bibitem{6425988}
A.~Zlotnik and S.~Li, ``Iterative ensemble control synthesis for bilinear
  systems,'' in \emph{Decision and Control (CDC), 2012 IEEE 51st Annual
  Conference on}, dec. 2012, pp. 3484 --3489.

\bibitem{periodic}
T.~Chambrion, ``Periodic excitations of bilinear quantum systems,''
  \emph{Automatica J. IFAC}, vol.~48, no.~9, pp. 2040--2046, 2012.

\bibitem{kato}
T.~Kato, \emph{Perturbation theory for linear operators}, ser. Classics in
  Mathematics.\hskip 1em plus 0.5em minus 0.4em\relax Berlin: Springer-Verlag,
  1995, reprint of the 1980 edition.

\bibitem{6426289}
U.~Boscain, M.~Caponigro, and M.~Sigalotti, ``Controllability of the bilinear
  schr{\"{o}}dinger equation with several controls and application to a 3d
  molecule,'' in \emph{Decision and Control (CDC), 2012 IEEE 51st Annual
  Conference on}, dec. 2012, pp. 3038 --3043.

\bibitem{noiesugny-CDC}
U.~Boscain, T.~Chambrion, P.~Mason, M.~Sigalotti, and D.~Sugny,
  ``Controllability of the rotation of a quantum planar molecule,'' in
  \emph{Proceedings of the 48th IEEE Conference on Decision and Control}, 2009,
  pp. 369--374.

\bibitem{ACCFEPS}
N.~Boussa{\"{i}}d, M.~Caponigro, and T.~Chambrion, ``Periodic control laws for
  bilinear quantum systems with discrete spectrum,'' in \emph{Proceedings of
  the American Control Conference}, 2012.

\bibitem{Morancey}
\BIBentryALTinterwordspacing
M.~Morancey, ``\BIBforeignlanguage{English}{Explicit approximate
  controllability of the schrödinger equation with a polarizability term},''
  \emph{\BIBforeignlanguage{English}{Mathematics of Control, Signals, and
  Systems}}, pp. 1--26, 2012. [Online]. Available:
  \url{http://dx.doi.org/10.1007/s00498-012-0102-2}
\BIBentrySTDinterwordspacing

\bibitem{CDCquadratic}
N.~Boussaid, M.~Caponigro, and T.~Chambrion, ``Approximate controllability of
  the {S}chr\"{o}dinger equation with a polarizability term,'' in
  \emph{Proceedings of the 51st IEEE Conference on Decision and Control},
  december 2012, pp. 3024 --3029.

\end{thebibliography}

\end{document}